\documentclass[10pt]{amsart}
\usepackage{amssymb,multicol}
\usepackage{mathtools}
\usepackage[pdfpagelabels=true,plainpages=false]{hyperref}
\usepackage[paperheight=11in,paperwidth=8.5in,
        top=1.0in,left=1.0in,right=1.00in,textheight=9.00in,]{geometry}
\usepackage{fancyhdr}
\usepackage{paralist}
\usepackage{cite}
\newcommand\bs{\par\bigskip}
\newcommand\bicond{\leftrightarrow}

\theoremstyle{plain}
\newtheorem{thm}{Theorem}[section]
\newtheorem{prop}[thm]{Proposition}
\newtheorem{lemma}[thm]{Lemma}
\newtheorem{prin}[thm]{Principle}
\newenvironment{thmLabeled}[1]
{\begin{thm}
\label{#1}{\thethm}
\label{page: #1}
}
{\end{thm}}

\newenvironment{prinLabeled}[1]
{\begin{prin}
\label{#1}{\theprin}
\label{page: #1}
}
{\end{prin}}

\newenvironment{lemmaLabeled}[1]
{\begin{lemma}
\label{#1}{\thelemma}
\label{page: #1}
}
{\end{lemma}}

\theoremstyle{definition}
\newtheorem{defn}[thm]{Definition}
\newcommand{\definedterm}[1]{\textit{#1}}
\newenvironment{defLabeled}[1]
{\begin{defn}
\label{#1}{\thedefn}
\label{page: #1}
}
{\end{defn}}

\theoremstyle{remark}
\newtheorem{obs}[thm]{Observation}
\newtheorem{notation}[thm]{Notation}

\begin{document}

\begin{center}
\textbf{Unlimited Category Theories for Mathematics are Inconsistent}

\bigskip
A Discussion of Michael Ernst's  \textit{The Prospects for Unlimited Category Theory:
Doing What Remains to be Done}
\cite{Ernst2015}

\bigskip
A Presentation to the Indiana University, Bloomington, Logic Seminar

\bigskip
William Wheeler
\end{center}

\bigskip
\section{Hopes and Efforts}

\bigskip
Proponents of category theory long hoped to escape the limits of set theory by founding mathematics 
on an unlimited category theory in which large categories, such as the category
\textit{ \bf Grp\/} of all groups, the category \textit{\bf Top\/} of all topological spaces, and
the category \textit{\bf Cat\/} of all categories would be (first-class) entities rather than just classes.
Penelope Maddy's expository paper \textit{What do we want a foundation to do?  Comparing set-theoretic,
category-theoretic, and univalent approaches\/}, \cite{Maddy2019}, pp. 17--22, discusses this in broad
terms.

In particular, in 1966, William Lawvere \cite{Lawvere1966} proposed the Category \textit{\bf Cat\/} of all 
Categories as A Foundation (CCAF) for mathematics.  This approach was improved by Colin McLarty in 1991 
\cite{McLarty1991}. CCAF is an axiomatization of properties of the category of all small 
categories, i.e., categories whose objects and morphisms are sets rather than proper classes.
Previously, in 1964, Lawvere \cite{Lawvere1964} had set forth the 
Elementary Theory of the Category of Sets (ETCS) as a possible foundation for mathematics.  
This also was subsequently improved by Colin McLarty \cite{McLarty1992}.
(For a current exposition of the proposal to found mathematics on ETCS, see \cite{Leinster2014}.)
In 1969, Solomon Feferman \cite{Feferman1969} proposed a system $S$ based 
on Willard van Orman Quine's New Foundations \cite{Quine1937}.  
Saunders Mac Lane \cite{MacLane1998} proposed a ``One Universe'' approach, building on Grothendieck's
Universes.

In addition to proposing the system $S$, Feferman \cite{Feferman1969} laid down three requirements,
which an axiomatic theory $S$ 
of categories should fulfill in order to be an \textit{unlimited category theory for mathematics}:
\begin{enumerate}
\item[(R1)] Form the category of all structures of a given kind, e.g., the category \textit{\bf Grp}
of all groups, \textit{\bf Top} of all topological spaces, and \textit{\bf Cat} of all categories, i.e., 
be unlimited.
\item[(R2)] Form the category $B^A$ of all functors from a category $A$ to a category $B$, where $A$ and
$B$ are any two categories.
\item[(R3)] Establish the existence of the natural numbers $N$, and carry out familiar operations on 
objects $a, b, \dots$ and collections $A, B, \dots$, including the formation of $\{a,b\}$, 
$(a,b)$, $A\cup B$, $A\cap B$, $A - B$, $A \times B$, $B^A$, $\bigcup A$, $\bigcap A$, 
$\prod_{x\in A}B_x$, etc.
\end{enumerate}
In 2013, Feferman \cite{Feferman2013} added as an explicit requirement one that had been an implicit 
requirement:
\begin{enumerate}
\item[(R4)] Establish the consistency of $S$ relative to a currently accepted system of set theory.
\end{enumerate}
He also added to the list of familiar operations in requirement (R3) the passage from a collection 
$A$ and equivalence relation $E$ on A to the canonical operation $p: A \rightarrow A/E$ where $p(x)$ is the 
equivalence class of $x$ with respect to $E$.

However, none of these efforts was ultimately successful.  In 1992, McLarty \cite{McLarty1992F}
showed that any theory
based on Quine's New Foundations, e.g., Feferman's system $S$ in \cite{Feferman1969}
would fail to be closed under Cartesian products.  The category theories CCAF and ETCS are not
unlimited, i.e., fail (R1).  Mac Lane's ``One Universe'' approach is equivalent to the existence
of a regular, inaccessible cardinal and so fails requirement (R4).

Nevertheless, Feferman in 2013  \cite{Feferman2013} still held out hope for an axiomatic
theory of unlimited categories, i.e., one fulfilling requirements (R1)--(R4).

\bs
\section{Denouement}

\bs
But, in 2015, Michael Ernst \cite{Ernst2015} proved: 

\begin{thm}If an axiomatic theory of categories satisfies
requirements (R1)--(R3), then it is inconsistent.
\end{thm}

His proof is an analogue, for a particular category $U$, of the proof, using Cantor's Diagonal Method, 
that there is no set $\mathcal{S}$ of all sets.

\bs
\subsection{The situation with sets.}

\strut
\bs
\begin{thmLabeled}{noSetofAllSets} (ZF) There is no set $\mathcal{S}$ of all sets.
\end{thmLabeled}

\begin{proof}
Suppose there were a set $\mathcal{S}$ of all sets.  Then, by the power set axiom, there would be
its exponential $2^{\mathcal{S}}$, i.e., its power set viewed as the collection of characteristic 
functions from $\mathcal{S}$ into the set $2 = \{0 = \emptyset,1 = \{\emptyset\}\}$.
For each set $C$, its characteristic function is $$\chi_C(x)=
\begin{cases}
1 &\text{if } x \in C\\
0 &\text{if } x \notin C\\
\end{cases}\quad .$$
\begin{enumerate}
\item [Step 1:] (Cantor's diagonal method) 
Suppose $f: \mathcal{S} \rightarrow 2^{\mathcal{S}}$. Let $A=\{x\in \mathcal{S} |
(f(x))(x)=0\}$.
By separation, $A$ is a set, so $A\in \mathcal{S}$.  Suppose there were a set $B$
such that $f(B) = \chi_A$.  Then  
$$\chi_A(B) = 1 \bicond B \in A \bicond (f(B))(B) =0 \bicond \chi_A(B) = (\chi_A)(B) = (f(B))(B) = 0 \ ,$$
where the last equation holds because $f(B)=\chi_A$.
But $\chi_A(B) = 1 \bicond \chi_A(B) = 0\neq 1$ is impossible.  
So there is no set $B$ such that $f(B) = \chi_A$.  
Thus, $f$ is not surjective, i.e., not onto.
\item [Step 2:]  
There is a mapping $f:\mathcal{S} \rightarrow 2^{\mathcal{S}}$ that is a surjection, i.e., onto.
Specifically, define $F:\mathcal{S} -> 2^{\mathcal{S}}$ by 
$F(x) = \chi_x$.
Because, for each $g \in 2^{\mathcal{S}}$, 
$g = \chi_{\{x: g(x) = 1\}}$ and $\{x: g(x) = 1\}$ is a set and so is a member of $\mathcal{S}$,
$$F(\{x: g(x)=1\}) = \chi_{\{x:g(x) =1\}} = g\, .$$
Thus, $F$ is a surjection, i.e., is onto.
\end{enumerate}
Steps 1 and 2 are mutually contradictory.  Therefore, there cannot be a set $\mathcal{S}$ of all sets.
\end{proof}

\bs
\subsection {The category \textit{\bf RGraph} of reflexive, directed graphs}

\strut
\bs
The category to which Ernst applies the preceding type of argument is the category 
\textit{\bf RGraph} of reflexive, directed graphs.  This category is discussed in \cite{Brown2008}.

\begin{defLabeled}{rgraph} A reflexive, directed graph (rgraph) consists of a collection of vertices
and a collection of (directed) edges such 
\begin{enumerate} 
\item each edge has a unique source vertex and a unique target vertex, and
\item each vertex $v$ has a distinguished edge $l_v$ such that the source and target of $l_v$
are both $v$.
\end{enumerate}
\end{defLabeled}

\begin{defLabeled}{rgraphmorphism} An rgraph morphism $M$ from an rgraph $G$ to an rgraph $H$ is a mapping
that takes the vertices of $G$ to vertices of $H$ and the edges of $G$ to edges of $H$ such that
\begin{enumerate}
\item for each edge $e$ of $G$ with source vertex $u$ and target vertex $v$, $M(e)$ is an edge of $H$ 
with source vertex $M(u)$ and target vertex $M(v)$, and
\item for each vertex $v$ of $G$, $M(l_v) = l_{M(v)}$.
\end{enumerate}
\end{defLabeled}

\begin{prop} The collections of rgraphs and rgraph morphisms is a category \textit{\bf RGraph} 
with the identity rgraph morphisms as the identity arrows.
\end{prop}

\begin{proof}
\begin{enumerate}
\item Each rgraph morphisms has a unique domain rgraph and a unique range rgraph.
\item Compositions of rgraph morphism are rgraph morphisms.
\item For each rgraph G, the identity rmorphism $1_G$ that maps each vertex of G to itself and each 
edge of $G$ to itself has the requisite categorical properties.
\end{enumerate}
\end{proof}

\begin{prop} The category \textit{\bf RGraph} is itself an rgraph with its objects as 
the vertices, its morphisms as the edges, and the identity rgraph morphisms as the distinguished edges.
\end{prop}

\bs
\subsection{Categorical Properties of the Category \textit{\bf RGraph}} 

\strut
\bs
The first issue is to identify the vertices and the edges of an rgraph categorically.

\begin{lemma} \textit{\bf RGraph} has a terminal object $1$, which is the rgraph whose vertex
is (the symbol) $!_1$ and whose only edge is the distinguished edge with source vertex $!_1$ and target
vertex $!_1$.
\end{lemma}

\begin{proof} For any rgraph $G$, the unique rgraph morphism from $G$ onto $1$ is the mapping that
takes every vertex of $G$ to $!_1$ and every edge of G to the edge of $1$ whose source and target
vertices are $!_1$.
\end{proof}

\begin{notation} The unique rgraph morphism from an rgraph $G$ to the terminal object $1$ will
be denoted at $!_G$.
\end{notation}

For each vertex $v$ of an rgraph $G$, there is a unique rgraph morphism from the terminal object
$1$ to $G$ that maps the vertex $!_1$ to $V$ and the unique edge of $1$ to the distinguished edge $l_v$
 of $G$ with source and target vertex $v$.  Conversely, each rgraph morphism from the terminal object
$1$ to $G$ has a unique vertex $v$ as the image of $!_1$ and the distinguished edge $l_v$ as the image
of the unique edge in $1$.

Accordingly, we can identify the vertices of $G$ with the rgraph morphisms of the terminal object $1$
into $G$.

In category theory, the morphisms from a terminal object of the category
are referred to as the \definedterm{global elements} of the category.

\begin{defLabeled}{E}Let $E$ be the rgraph with vertices $\mathcal{s}$ and $\mathcal{t}$, distinguished edges $l_{\mathcal{s}}$, whose source and target vertices are $\mathcal{s}$, and $l_{\mathcal{t}}$, whose source and target vertices are $\mathcal{t}$, and an edge $1_E$ whose source vertex is $\mathcal{s}$ 
and whose target vertex is $\mathcal{t}$.
\end{defLabeled}

For each edge $e$ of an rgraph $G$, there is a unique rgraph morphism from $E$ to $G$ that maps the edge $1_E$ to e.  Conversely, each rgraph morphism from $E$ to $G$ determines a unique edge of $G$, namely,
the image of $1_E$.  

Accordingly, we can identify the edges of an rgraph $G$ with the rgraph morphisms from $E$ to $G$.

\vtop{
\begin{defLabeled}{SpecialGraphs}
\begin{enumerate}
\item An rgraph $G$ is said to be \definedterm{complete} if, for each pair of vertices $u$ and $v$, there is exactly one
edge with source vertex $u$ and target vertex $v$, i.e., or equivalently, there is a unique
rgraph morphism $e$ from the rgraph $E$  to $G$ such such that $e(\mathcal{s}) = u$ and 
$e(\mathcal{t}) = v$.
\item An rgraph $G$ is said to be a \definedterm{tournament} if, for each pair of vertices $u$ and $v$,
there is exactly one edge between them, i.e., there is a unique rgraph morphism $e$ from the rgraph $E$
to $G$ such that $e(\mathcal{s}) = u$ and $e(\mathcal{t})=v$ or that
$e(\mathcal{s})=v$ and $e(\mathcal{t}) =u$.
\item An rgraph $G$ is said to be discrete if the only edges are the distinguished edges for its
vertices.
\end{enumerate}
\end{defLabeled}
}

\begin{prinLabeled}{AxiomOfChoiceForDiscrete}(Axiom of Choice for Discrete Graphs)
If $f$ is an rgraph morphism from G into a discrete rgraph $D$, then there is an
rgraph morphism from $D$ into $G$ such that $f \circ h \circ f = f$.
\end{prinLabeled}

We shall assume that the Axiom of Choice for Discrete Graphs holds in the category
\textit{\bf RGraph}.

\begin{lemma} The category \textit{\bf RGraph} has products.
\end{lemma}

\begin{lemma} The category \textit{\bf RGraph} has exponentials.
\end{lemma}

\begin{defLabeled}{truthvaluegraph} The \definedterm{truth value graph} for the category \textit{\bf RGraph}
is the graph with two vertices \textit{true} and \textit{false}, the distinguished edge $l_{\textit{true}}$
whose source and target vertices are \textit{true}, the distinguished edge
$l_{\textit{false}}$ whose source and target vertices are \textit{false}, another edge whose source and
target vertices are \textit{true}, and an edge whose source is \textit{true} and whose target is 
\textit{false}, and an edge whose source is \textit{false} and whose target is \textit{true}.
\end{defLabeled}

\begin{lemma} The aforedefined truth value graph is a subobject classifier in the category \textit{\bf 
Rgraph}.
\end{lemma}

\begin{obs} The category \textit{\bf RGraph} is a topos.
\end{obs}

\bs
\subsection{The Contradiction}

\strut\bs
\begin{defLabeled}{GraphK2}
Define the rgraph $K_2$ to be the complete rgraph with vertices $\mathcal{1}$ and $\mathcal{2}$; 
the distinguished edges $l_{\mathcal{1}}$, whose source and target vertices are $\mathcal{1}$,
and $l_{\mathcal{2}}$, whose source and target vertices are $\mathcal{2}$;  a unique edge
whose source vertex is $\mathcal{1}$ and whose target vertex is $\mathcal{2}$; and a unique
edge whose source vertex is $\mathcal{2}$ and whose target vertex is $\mathcal{1}$.
\end{defLabeled}

Recall that the category \textit{\bf RGraph} considered as an rgraph is denoted by $U$.

The rgraph $K_2^U$ will be the rgraph that is the analogue of $2^{\mathcal{S}}$ in the proof
of Theorem \ref{noSetofAllSets} for a derivation of a contradiction.

\begin{thmLabeled}{Contradiction} The Contradiction:
\begin{enumerate}
\item There is no rgraph morphism $F: U \rightarrow K_2^U$ that is onto for global elements.
\item There is an rgraph morphism $F: U \rightarrow K_2^U$ that is onto for global elements.
\end{enumerate}
\end{thmLabeled}

\begin{proof} (A sketch the proof; )
\begin{enumerate} 
\item [Step 1: ] The proof of the first assertion uses Corollary 1.2 of Lawvere's
1969 paper \textit{Diagonal arguments and cartesian closed categories} \cite{Lawvere1969}:

\begin{thm} (Lawvere) In any topos $\mathcal{C}$, the following holds:  For any 
$\mathcal{C}$-object $B$, if there exists a $\mathcal{C}$-arrow $f:B \rightarrow B$ such that
$f \circ b \neq b$ for all $b: 1 -> B$, then for no $\mathcal{C}$-object $A$ does there exist
a $\mathcal{C}$-arrow $A \rightarrow B^A$ that is onto for global elements.
\end{thm}

Let $p_{K_2} : K_2 \rightarrow K_2$ be the rgraph morphism that permutes the vertices of $K_2$.
Then $p_{K_2}$ satisfies then $p_{K_2} \circ b \neq b$ for any $b: 1 -> K_2$.  By Lawvere's theorem,
there is no \textit{\bf RGraph} morphism $F: U \rightarrow K_2^U$ that is onto for global elements.

\item [Step 2: ]  This step is considerably longer.  Two lemmas are needed; they are proved in the
appendices of Ernst's paper \cite{Ernst2015}.

\begin{lemmaLabeled}{ExponentialCompleteness} (Exponential Completeness) If $G$ is a complete rgraph,
then $G^H$  is complete for any rgraph $H$.
\end{lemmaLabeled}

Because the rgraph $K_2$ is complete, $K_2^U$ is complete, i.e., for each pair of vertices $u$ and $v$
of $K_2^U$, there is a unique edge with source vertex $u$ and target vertex $v$.  Denote this unique edge
by $\prec u, v\succ$.

\begin{lemmaLabeled}{Tournament} (Tournament) For any rgraph $G$, for each vertex x of $K_2^G$, there
is a tournament $Q_x$ that is a subrgraph of $K_2^G$.  Also, for distinct vertices $x$ and $y$, 
$Q_x \neq Q_y$.
\end{lemmaLabeled}

\begin{lemmaLabeled}{StrengthenedTournament}(StrengthenedTournament) For any rgraph $G$, there is a 
mapping $I:K_2^G \rightarrow U$ such that, for each vertex $x$ of $K_2^G$, $I(x)$ is a tournament
that is a subrgraph of $K_2^G$.  Furthermore, for distinct vertices $x$ and $y$, $I(x) \neq I(y)$.
\end{lemmaLabeled}

Let $J$ be such a mapping $I$ for $G=U$, $J:K_2^U \rightarrow U$.

Let $v_{\text{other}}$ be a vertex of $K_2^U$.

Define an \textit{\bf RGraph} morphism $F: U \rightarrow K_2^U$ by 
$$F(x) = 
\begin{cases}
y&\text{if  $x$ is a vertex of $U$ and,  for some vertex $y$ of $K_2^U$, } x = J(y)\\
v_\text{Other}&\text{if $x$ is a vertex of $U$ and there is no vertex $y$ of $K_2^U$ such that } x=J(y)\\
\prec F(u),F(v)\succ&\text{if $e$ is an edge with source vertex $u$ and target vertex $v$}\\
\end{cases}
$$
For each vertex $y$ in $K_2^U$, $F(J(y))=y$.  
Thus, $F$ is onto for global elements.
\end{enumerate}
\end{proof}

\begin{thmLabeled}{NoTheory}
If an axiomatic theory of categories fulfills Feferman's requirements (R1), (R2), and (R3),
then it is inconsistent.
\end{thmLabeled}

\begin{proof}  If an axiomatic theory of categories fulfills Feferman's requirements (R1), (R2), and
(R3), then it can form the category \textit{\bf RGraph} and prove the preceding Theorem \ref{Contradiction},
i.e., it proves $\neg \varphi \wedge \varphi$, where $\varphi$ is the the statement \textit{There is an 
rgraph morphism $F: U \rightarrow K_2^U$ that is onto for global elements}.  Thus, the theory is inconsistent.
\end{proof}

\bs
\section {Going Forward}

\strut\bs

Ernst points out that, prior to his paper, one of the criticisms leveled against various, limited foundations,
in particular, ZFC, was that they did not provide a foundation for unlimited category theories.
But, because of his proof that unlimited category theories are inconsistent, nothing could have
provided a foundation for them.  This cancels this criticism of ZFC and comparable foundations.

Penelope Maddy, in her insightful paper \textit{What do we want a foundation to do} \cite{Maddy2019}
analyzes the various positions regarding the foundations of mathematics
and concludes that there isn't a well-defined notion of what a foundation is but rather 
that there are well-defined notions 
of what a foundation should do.  She believes that the fruitful way to proceed with foundational matters 
is to consider what a foundation should or might do and the benefits it might provide.  For set theory, she
identifies the benefits that it provides as `` a \textbf{Generous Arena} where all of modern mathematics
takes place side-by-side'' with sharing from one part to another, ``a \textbf{Shared Standard} 
of what counts as a legitimate construction or proof'', \textbf{Risk Assessment} of the level
of risk of inconsistency in theories and methods, and a ``\textbf{Meta-mathematical Corral} so that
formal techniques can be applied to all of mathematics at once.'' 
Given Ernst's proof that unlimited category theory
is inconsistent, Maddy identifies the benefit of category theory to be ``\textbf{Essential Guidance}'' 
to ``guide mathematicians toward the important [mathematical] structures and characterize them
strictly in terms of their mathematically essential features.... but only for those branches of 
mathematics of roughly algebraic character.''  In the latter regard, she notes that Saunders Mac Lane
himself (in \cite{MacLane1986}, p. 407) states ``Categories and Functors are everywhere in topology 
and in parts of algebra, but they do not yet relate very well to most of analysis.''
For the more recent \textit{univalent foundations}, Maddy sees the primary benefit it provides as being
\textbf{Proof Checking}.  She views all of these benefits as important benefits which no one ``foundation'' 
could provide.

John Baldwin \cite{Baldwin2021}, building on Maddy's work, proposes the position that, as regards
the \textbf{Essential Guidance} that Category Theory provides, Category Theory should not be compared
to Set Theory as a foundation but rather be compared to Model Theory as a \textbf{scaffold} within
which to do mathematics.  He identifies as other scaffolds Descriptive Set Theory, which organizes
``many subjects that are captured by Polish Spaces'' and ``the Langlands Program that expands on
Weil's Rosetta stone to unify number theory with harmonic analysis (and points in between)''.

\bigskip

\bigskip

\bibliography{Ernst}{}
\bibliographystyle{plain}
\end{document}